\documentclass[a4paper,11pt]{article}
\usepackage[utf8]{inputenc}
\usepackage[english]{babel}
\usepackage{amsthm}
\usepackage{amsmath}
\usepackage{setspace}
\usepackage{amssymb}
\usepackage{amscd}
\usepackage{units}
\usepackage[pdftex]{graphicx}
\usepackage{geometry}
\DeclareGraphicsExtensions{.pdf,.png,.jpg}
\newcommand{\ere}{\mathbb{R}}

\newcommand{\C}{\mathcal{C}}
\newcommand{\Ha}{\mathcal{H}}

\newcommand{\gen}[1]{\left\langle#1\right\rangle}
\newcommand{\norm}[1]{\left\lVert#1\right\rVert}
\newcommand{\llaves}[1]{\left\{#1\right\}}
\newcommand{\seg}[1]{\left[#1\right]}
\newcommand{\rayo}[1]{\left[#1\right\rangle}
\newcommand{\parent}[1]{\left(#1\right)}
\theoremstyle{remark}
 \theoremstyle{definition}
  
  \providecommand{\definitionname}{Definition}
\theoremstyle{plain}
 \newtheorem{teor}{\protect\theoremname}[section]
  \providecommand{\theoremname}{Theorem}
\theoremstyle{plain}
 \newtheorem*{teor*}{\protect\theoremname}
  \providecommand{\theoremname}{Theorem}
\theoremstyle{plain}
 \newtheorem{lema}[teor]{\protect\lemmaname}
   \providecommand{\lemmaname}{Lemma}
  \theoremstyle{plain}
   \newtheorem*{lema*}{\protect\lemmaname}
   \providecommand{\lemmaname}{Lemma}
  \theoremstyle{plain}
   
   \providecommand{\propositionname}{Proposition}
  \theoremstyle{plain}
   \newtheorem{cor}[teor]{\protect\corolariumname}
   \providecommand{\corolariumname}{Corollary}
\numberwithin{equation}{section}
\title{\textbf{Orthocentric Systems in Minkowski Planes}}
\author{Tob\'ias Rosas Soto (tjrosas@hotmail.com)\\
Wilson Pacheco Redondo (wpachecoredondo@gmail.com)\\
Departamento de Matemática\\
Facultad Experimental de Ciencias\\
Universidad del Zulia\\
Venezuela}

\begin{document}

\maketitle

\begin{abstract}
A new way to define the notion of $\C$-orthocenter will be displayed by studying some
propierties of four points in the plane which allows to extend the notion of Euler's line,
the Six Point Circles and the three-circles theorem, for normed planes. In the
present paper (which can be regarded as extention of \cite{M-WU} in Minkowski plane in
general) we derive several characterizations of the Euclidianity of the plane.\\

\end{abstract}

\section{Introduction}
By $(M,\norm{\circ})$ we denote an arbitrary Minkowski plane (i.e., a real
two-dimensional normed linear space) with unit circle $\C$, origin $O$, and norm
$\norm{\circ}$. \mbox{Basic} \mbox{references} to the geometry of Minkowski planes are
\cite{M-SW-WE}, \cite{M-SW},\cite{WU} and the \mbox{monograph} \cite{T}. For any point
$x\in(M,\norm{\circ})$ and any number $\lambda> 0$, the set $\C(x,\lambda) := x + \lambda\C$
is said to be the circle centered at $x$ and having radius $\lambda$. It has been shown by
E. Asplund and B. Gr\"{u}nbaum in \cite{A-G} that the following theorem, which is the
extension of the classical three-circles theorem in the euclidean plane, also holds in
strictly convex planes.

\begin{teor*}
If three circles $\C(x_{1},\lambda)$, $\C(x_{2},\lambda)$, and $\C(x_{3},\lambda)$ pass
through a common point $p_{4}$ and intersect pairwise in the points $p_{1}$, $p_{2}$, and
$p_{3}$, then there exists a circle $\C(x_{4},\lambda)$ such that 
$\{p_{1},p_{2},p_{3}\}\subseteq\C(x_{4},\lambda)$.
\end{teor*}

The point $p_{4}$, in the above theorem, is called the $\C$-orthocenter of the triangle
$\triangle p_{1} p_{2} p_{3}$, and it is also evident that $p_{i}$ is the $\C$-orthocenter of
the triangle $\triangle p_{j} p_{k} p_{l}$, where $\{i, j, k, l\} = \{1, 2, 3, 4\}$.
For these reason the set of four points $\{p_{1},p_{2},p_{3},p_{4}\}$ is called a
$\C$-orthocentric system. Then, it is clear that the set of the circumcenters $x_{l}$ of the
triangles $\triangle p_{i}p_{j}p_{k}$ is a $\C$-orthocentric system too (see Theorem 2.9 in
\cite{M-SP}).

The notion of $\C$-orthocenter is define in strictly convex Minkowski plane because the
strictly convex property guarantees that there exists only one circle passing through the
three different points, i.e, each triangle with circumcenter has one and only one
\mbox{circumcenter}. Note that the way Asplund and Gr\"{u}nbaum used to define the notion of
$\C$-orthocenter needs three different circles of the circumcircle. However, there are cases
where no three different circles exist of the circumcircle that holds the hypothesis
of the previous theorem and therefore, in this situation we have a problem with this way to
define the $\C$-orthocenter. Another problem with the definition of Asplund and
Gr\"{u}nbaum is that there are cases where the intersection of the three circles is more
than one point and therefore, how to known which point is the $\C$-orthocenter of the
triangle.

On the other hand, in Minkowski planes there are cases where a triangle can have many
\mbox{circumcenters} or not have it and therefore, the same happen with the
\mbox{$\C$-orthocenters}, but as our intention is to characterize the euclidianity in
Minkowski planes, by studying some geometric properties of the $\C$-orthocentric systems, we
will \mbox{focus} on defining the notion of $\C$-orthocenter for triangles having
circumcenter. The problem here is how to decide which \mbox{$\C$-orthocenter} is associates
with a particular circumcenter, when the triangle has many circumcenters.

Note that in the above theorem one can see that $x_{4}$ is the circumcenter of the triangle
$\triangle p_{1}p_{2}p_{3}$ and the points $x_{1},x_{2},x_{3}$, centers of the circles, form
a triangle where $p_{4}$ is its circumcenter. Furthermore, we can see that the triangles
$\triangle p_{1}p_{2}p_{3}$ and $\triangle x_{1}x_{2}x_{3}$ are symmetric with respect to
one point $q$ (see \textit{Figure \ref{FG1}}). Following this ideas we will introduce in the
section three a new way to define the notion of $\C$-orthocenter and $\C$-orthocentric
systems.

\begin{figure}[h!]
 \begin{center}
  \includegraphics[scale=0.22]{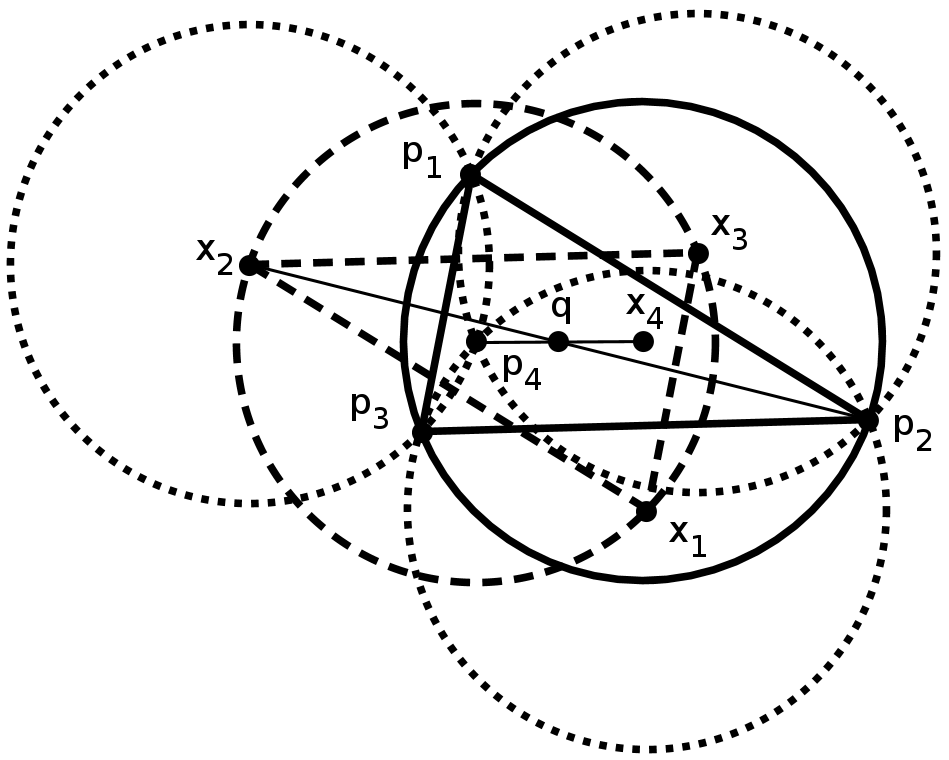}\; \includegraphics[scale=0.4]{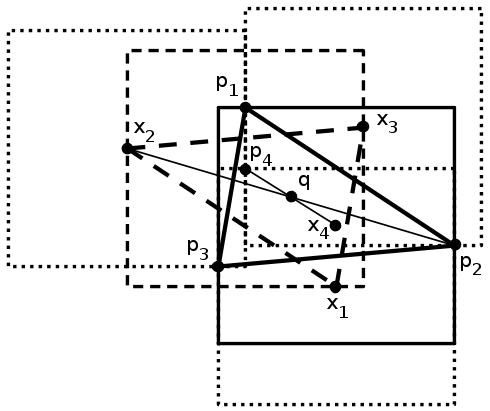}
  \caption{\label{FG1} \textit{$\C$-orthocenter}}
 \end{center}
\end{figure}

%%%%%%%%%%%%%%%%%%%%%%%%%%%%%%%%%%%%%%%%%%%%%%%%%%%%%%%%%%%%%%%%%%%%%%%%%%%%%%%%%%%%%%%%%%%%
%%%%%%%%%%%%%%%%%%%%%%%%%%%%%%%%%%%%%%% Section 1 %%%%%%%%%%%%%%%%%%%%%%%%%%%%%%%%%%%%%%%%%%
%%%%%%%%%%%%%%%%%%%%%%%%%%%%%%%%%%%%%%%%%%%%%%%%%%%%%%%%%%%%%%%%%%%%%%%%%%%%%%%%%%%%%%%%%%%%

\section{Notations and Some Lemmas}

For our discussion, define
\mbox{$\C(\triangle abc)=\llaves{x\in M:\norm{x-a}=\norm{x-b}=\norm{x-c}}$}
with $M$ a Minkowski plane, as the set of circumcenters of the triangle $\triangle abc$.
On the other hand, we denote by $H_{p,-2}$ and $S_{p}$, the  homothety with center $p$
and ratio $-2$ and the the symmetry with center $p$ respectively, they are defined by
$H_{p,-2}(w)=3p-2w$ and $S_{p}(w)=2p-w$. Note that $S_{p}(w)-S_{p}(v)=v-w$. Thus, the
symmetries are isometries in Minkowski planes, i.e., $\norm{S_{p}(w)-S_{p}(v)}=\norm{w-v}$
for all $w,v\in M$.

For $x\neq y$, we denote by $\gen{x,y}$ the line passing through $x$ and $y$, by $\seg{x,y}$
the segment between $x$ and $y$, and by $\rayo{x,y}$ the ray with starting point $x$ passing
through $y$. For any two different points $x,y\in\C(w,\lambda)$ with $y\neq x$, the
line $\gen{x,y}$ divides the plane into two half-planes, $L_{+}$ and $L_{-}$, and therefore
it divides $\C(w,\lambda)$ into two arcs between the points $x$ and $y$,
$\C(w,\lambda)\cap L_{+}$ and $\C(w,\lambda)\cap L_{-}$. They will be denote by
$Arc^{+}_{\C(w,\lambda)}(x,y)$ and $Arc^{-}_{\C(w,\lambda)}(x,y)$. We
need the definitions of Isosceles orthogonality, Birkhoff orthogonality, and Busemann
angular bisectors. Let $x,y\in(M,\norm{\circ})$. The point $x$ is said to be isosceles
orthogonal to $y$ if $\norm{x+y}=\norm{x-y}$, and in this case we write $x \bot_{I}y$
(cf. \cite{AL}). On the other hand, $x$ is said to be Birkhoff orthogonal to $y$ if
$\norm{x+ty}\geq\norm{x}$ holds for all $t\in\ere$, and for this we write $x\bot_{B}y$. We
refer to \cite{A} and \cite{AL} for basic properties of isosceles orthogonality and Birkhoff
orthogonality, and the relations between them.

For non-collinear rays $\rayo{p,a}$ and $\rayo{p,b}$, the ray
\begin{equation*}
 \rayo{p,\frac{1}{2}\left(\frac{a-p}{\norm{a-p}}+\frac{b-p}{\norm{b-p}}\right)+p}
\end{equation*}
is called the Busemann angular bisector of the angle spanned by $\rayo{p,a}$ and
$\rayo{p,b}$, and it is denoted by $A_{B}\parent{\rayo{p,a},\rayo{p,b}}$ (cf. \cite{B},
\cite{D2}). It is trivial to see that when $\norm{a-p}=\norm{b-p}$, then
\begin{equation*}
 A_{B}\parent{\rayo{p,a},\rayo{p,b}}=\rayo{p,\frac{a+b}{2}}
 \end{equation*}

The followings lemmas are needed for our investigation.

\begin{lema}\label{lepre1}
 (cf. [9, (2.4)])If for any $x,y\in(M,\C)$ with $x\bot_{I}y$ there exists a number $t>1$
 such that $x\bot_{I}ty$, then $(M,\C)$ is Euclidean.
\end{lema}

\begin{lema}\label{lepre2}(cf. [1, (10.9)])
 A Minkowski plane $(M,\norm{\circ})$ is Euclidean if and only if the implication
 \begin{equation*}
  x\bot_{I}y\Longrightarrow x\bot_{B}y
 \end{equation*}
holds for all $x,y\in M$.
\end{lema}

\section{Main Results}
In this section we present the main results, one of them will allow us to introduce the new
way to define the notion of $\C$-orthocenter and therefore $\C$-orthocentric systems. The
other will be used as support to show the results about characterizations of ecuclidean
planes among normed planes, by studying geometric properties of $\C$-orthocentric systems.

%%%%%%%%%%%%%%%%%%%%%%%%%%%%%%%%%%%%%%%%%%%%%%%%%%%%%%%%%%%%%%%%%%%%%%%%%%%%%%%%%%%%%%%%%%%%
%%%%%%%%%%%%%%%%%%%%%%%%%%%%%%%%%%%%%%% Teorema 1 %%%%%%%%%%%%%%%%%%%%%%%%%%%%%%%%%%%%%%%%%%%%%
%%%%%%%%%%%%%%%%%%%%%%%%%%%%%%%%%%%%%%%%%%%%%%%%%%%%%%%%%%%%%%%%%%%%%%%%%%%%%%%%%%%%%%%%%%%%

\begin{teor} \label{teor1}
 Let $M$ be a Minkowski plane. Let $x_{1}$, $x_{2}$, $x_{3}$ and $p_{4}$ be points in $M$.
 Let \mbox{$m_{1}$, $m_{2}$ and $m_{3}$} be the midpoints of the segments
 $\seg{x_{2},x_{3}}$, $\seg{x_{1},x_{3}}$ and $\seg{x_{1},x_{2}}$ \mbox{respectively}.
 Define the points $p_{i}=S_{m_{i}}(p_{4})$, for $i=1,2,3$. Then the following holds:
 \begin{enumerate}
  \item The segments $\seg{x_{i},p_{i}}$ has the same midpoint $q$, for $i=1,2,3$. Also
  \mbox{$2(q-m_{i})=x_{i}-p_{4}$} for $i=1,2,3$, i.e,
  \begin{equation*}
   q=\frac{x_{1}+x_{2}+x_{3}-p_{4}}{2}.
  \end{equation*}
  \item If $x_{4}=S_{q}(p_{4})$, then $x_{i}-x_{j}=p_{j}-p_{i}$ for
  $\llaves{i,j}\subset\llaves{1,2,3,4}$.
  \item $x_{i}-p_{j}=p_{k}-x_{l}$, where $\llaves{i,j,k,l}=\llaves{1,2,3,4}$.
  \item If $g=\frac{x_{1}+x_{2}+x_{3}}{3}$, then $H_{g,-2}(p_{4})=x_{4}$.
 \end{enumerate}
\end{teor}
\begin{proof}
 (1) Since $m_{i}=\frac{x_{j}+x_{k}}{2}$ for $\{i,j,k\}=\{1,2,3\}$ and
 \begin{equation}
 p_{i}=S_{m_{i}}(p_{4})=x_{j}+x_{k}-p_{4} \label{eq1}
 \end{equation}
 then
 \begin{eqnarray}
  p_{i}+x_{i}=x_{i}+x_{j}+x_{k}-p_{4} \label{eq2}
 \end{eqnarray}
 for  $\{i,j,k\}=\{1,2,3\}$. Thus, 
 \begin{equation*}
  q=\frac{x_{1}+x_{2}+x_{3}-p_{4}}{2}.
 \end{equation*}
 satisfied the propertie (see \textit{Figure \ref{FG7}}).

 (2) If $x_{4}=S_{q}(p_{4})$, since $x_{i}=S_{q}(p_{i})$ for $i=1,2,3$, then 
 \begin{equation} \label{eq3}
  x_{i}-x_{j}=S_{q}(p_{i})-S_{q}(p_{j})=p_{j}-p_{i}
 \end{equation}
 for $\{i,j\}\subset\{1,2,3,4\}$
 
 \begin{figure}[h!]
  \begin{center}
   \includegraphics[scale=0.14]{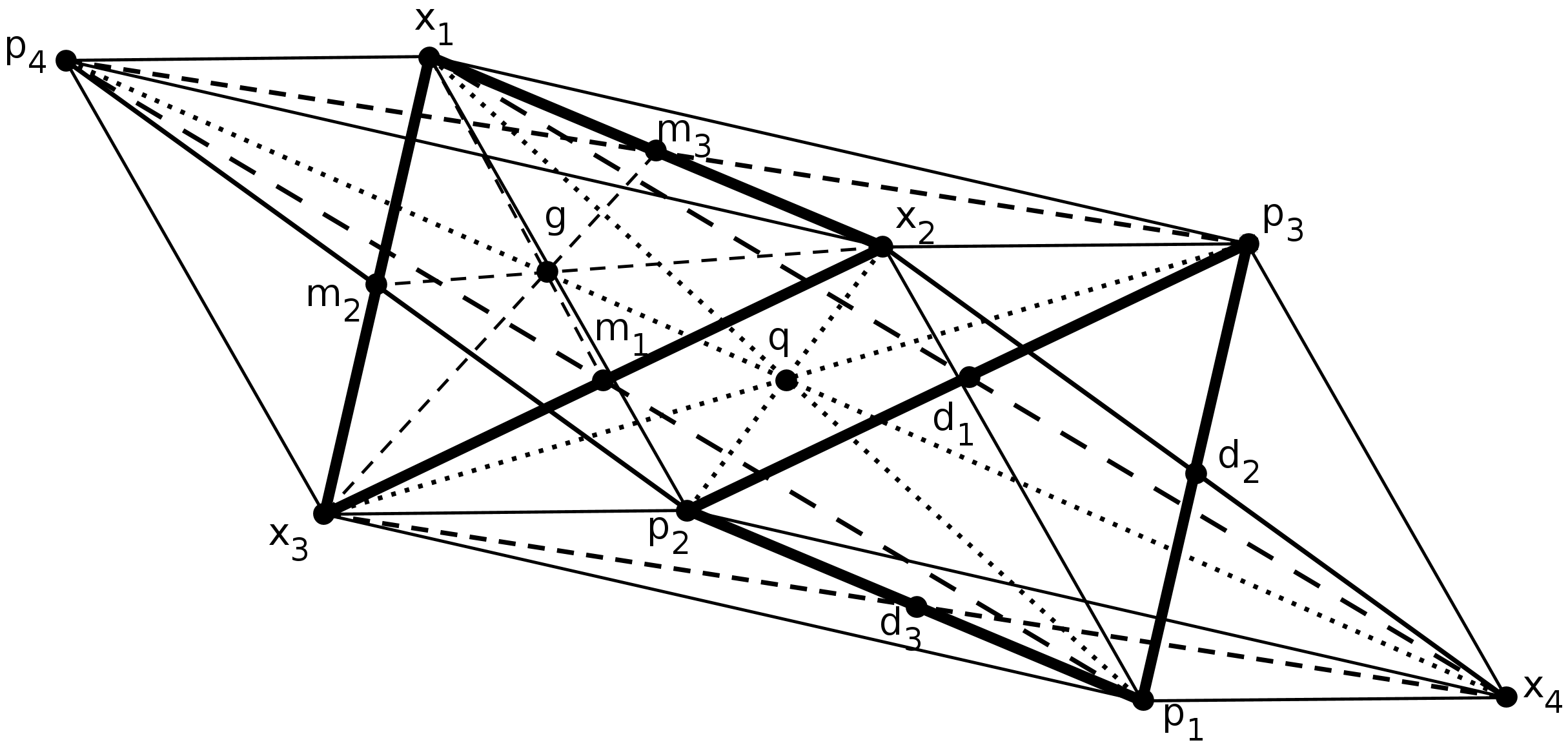}
   \caption{\label{FG7} \textit{Demonstration Theorem \ref{teor1}}}
  \end{center}
 \end{figure}
 
 (3) Let $d_{i}=S_{q}(m_{i})$ for $i=1,2,3$, then $d_{i}$ are the midpoints of the
 segments $\seg{p_{j},p_{k}}$ for $\{i,j,k\}=\{1,2,3\}$. Then, using the equation (\ref{eq1})
 and the simmetry $S_{q}$ we have
 \begin{equation}\label{eq4}
 x_{i}=p_{j}+p_{k}-x_{4}
 \end{equation}
 for $\llaves{i,j,k}=\llaves{1,2,3}$. From equations (\ref{eq3}) and (\ref{eq4}), then we get
 the propiertie desired.

 (4) Since $g=\frac{x_{1}+x_{2}+x_{3}}{3}$ and by the homothety definition, then
 \begin{equation*}
   H_{g,-2}(p_{4})=3g-2p_{4}=3\left(\frac{x_{1}+x_{2}+x_{3}}{3}\right)-2p_{4}=x_{4}
 \end{equation*}
\end{proof}

By the construction shown in Theorem (\ref{teor1}), we will call \textit{$p_{4}$-Antitriangle}
of the \mbox{triangle} $\triangle x_{1}x_{2}x_{3}$ to the triangle $\triangle p_{1}p_{2}p_{3}$, i.e,
the \textit{Antitriangle} with respect to the point $p_{4}$. Now, we can introduce the new
way to define the notion of $\C$-orthocenter saying: given the points $x_{1}$, $x_{2}$,
$x_{3}$, and $p_{4}$ in $\C\parent{\triangle x_{1}x_{2}x_{3}}$, we will
say that $x_{4}$ is the $\C$-orthocenter of the triangle $\triangle x_{1}x_{2}x_{3}$
associated with $p_{4}$ if $S_{q}(p_{4})=x_{4}$, where $q$ is the point of symmetry of the
triangle $\triangle x_{1}x_{2}x_{3}$ and its \textit{$p_{4}$-Antitriangle} as we defined in
the Theorem (\ref{teor1}). Furthermore, if we have four points $p_{1}$, $p_{2}$, $p_{3}$ and
$p_{4}$ in the plane. We will say that $\{p_{1},p_{2},p_{3},p_{4}\}$ is a $\C$-orthocentric
system if there exist a circumcenter $x_{i}$ of the triangle $\triangle p_{j}p_{k}p_{l}$,
with $\{i,j,k,l\}=\{1,2,3,4\}$ such that $p_{i}$ is the $\C$-orthocenter of the triangle
$\triangle p_{j}p_{k}p_{l}$. Then, let
\mbox{$\Ha(\triangle abc)=\llaves{x\in M: x\text{ \textit{is a $\C$-orthocenter of the
triangle}}}$} be the set of $\C$-orthocenters of the triangle $\triangle abc$.

As we said at the beginning, $x_{4}$ is the circumcenter of the $p_{4}$-Antitriangle of
the triangle $\triangle x_{1}x_{2}x_{3}$ and it holds the definition given by E. Asplund and
B. Gr\"{u}nbaum in \cite{A-G}. In this way one can see the relation between the
$\C$-orthocenters and the circumcenters of the triangle $\triangle x_{1}x_{2}x_{3}$ and we
show the uniqueness of the $\C$-orthocenter, given the circumcenter of the triangle.

%%%%%%%%%%%%%%%%%%%%%%%%%%%%%%%%%%%%%%%%%%%%%%%%%%%%%%%%%%%%%%%%%%%%%%%%%%%%%%%%%%%%%%%%%%%%
%%%%%%%%%%%%%%%%%%%%%%%%%%%%%%%%%% Corolario 1 %%%%%%%%%%%%%%%%%%%%%%%%%%%%%%%%%%%%%%%%%%%%%
%%%%%%%%%%%%%%%%%%%%%%%%%%%%%%%%%%%%%%%%%%%%%%%%%%%%%%%%%%%%%%%%%%%%%%%%%%%%%%%%%%%%%%%%%%%%

\begin{cor}\label{corl1}
With the assumptions of the previous theorem. Then the following holds:
 \begin{enumerate}
  \item If $x_{4}=S_{q}(p_{4})$, then $x_{4}=x_{1}+x_{2}+x_{3}-2p_{4}$
  \item $m_{i}=\frac{p_{i}+p_{4}}{2}$ for all $i=1,2,3$.
  \item If $g=\frac{x_{1}+x_{2}+x_{3}}{3}$ and $g_{1}=\frac{p_{1}+p_{2}+p_{3}}{3}$, then
  $g_{1}=S_{q}(g)$ and $S_{g}(p_{4})=g_{1}$.
 \end{enumerate}
 \end{cor}
 
\begin{proof}
 (1) By Theorem (\ref{teor1}) $q=\frac{x_{1}+x_{2}+x_{3}-p_{4}}{2}$ and, since
 $q=\frac{x_{4}+p_{4}}{2}$, then \mbox{$x_{4}=x_{1}+x_{2}+x_{3}-2p_{4}$}.
 
 (2) Since $m_{i}=\frac{x_{j}+x_{k}}{2}$ for $\{i,j,k\}=\{1,2,3\}$, then by $(3)$ from
 Theorem (\ref{le1}) (ver \textit{Figura \ref{FG9}}), we get 
 \begin{equation*}
  m_{i}=\frac{x_{j}+x_{k}}{2}=\frac{p_{i}+p_{l}}{2}
 \end{equation*}
 for $\{i,j,k,l\}=\{1,2,3,4\}$.

\begin{figure}[h!]
 \begin{center}
  \includegraphics[scale=0.12]{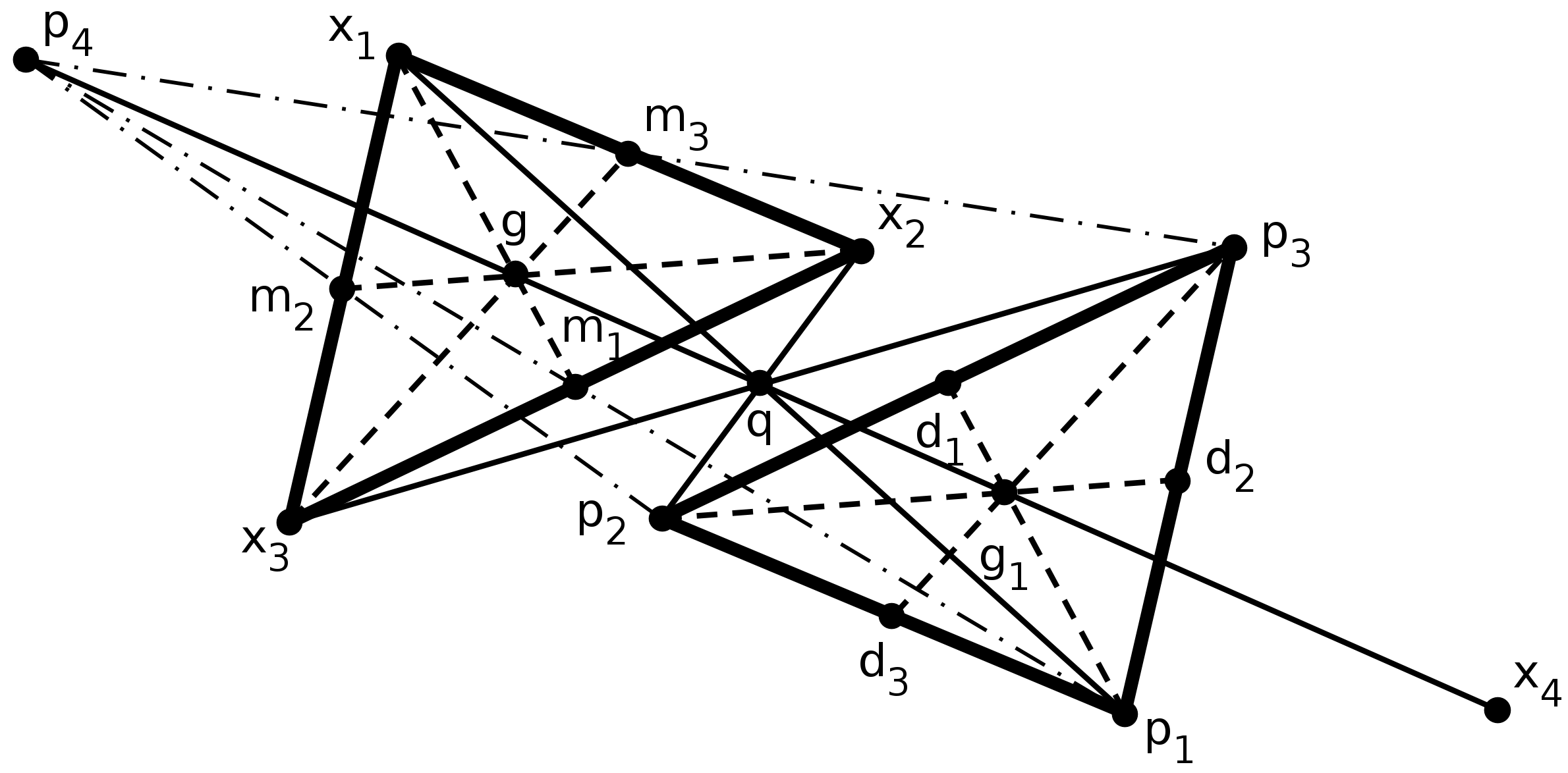}
  \caption{\label{FG9} \textit{Demonstration Corollary \ref{corl1}}}
 \end{center}
\end{figure}
 
(3) Since $p_{i}=S_{m_{i}}(p_{4})$ for $i=1,2,3$ (see \textit{Figure \ref{FG9}}), then
 \begin{equation*}
  g_{1}=\frac{p_{1}+p_{2}+p_{3}}{3}=\frac{2x_{1}+2x_{2}+2x_{3}-3p_{4}}{3}=
  2g-p_{4}=S_{g}(p_{4})
 \end{equation*}
 
 On the other hand, since $g_{1}=\frac{p_{1}+p_{2}+p_{3}}{3}$,
 \begin{equation*}
  g_{1}=\frac{2(m_{1}+m_{2}+m_{3})-3p_{4}}{3}=
  x_{1}+x_{2}+x_{3}-p_{4}-\frac{x_{1}+x_{2}+x_{3}}{3}=S_{q}(g)
 \end{equation*}
 \end{proof}

%%%%%%%%%%%%%%%%%%%%%%%%%%%%%%%%%%%%%%%%%%%%%%%%%%%%%%%%%%%%%%%%%%%%%%%%%%%%%%%%%%%%%%%%%%%%
%%%%%%%%%%%%%%%%%%%%%%%%%%%%%%%%%%%%% Corolario 2 %%%%%%%%%%%%%%%%%%%%%%%%%%%%%%%%%%%%%%%%%%
%%%%%%%%%%%%%%%%%%%%%%%%%%%%%%%%%%%%%%%%%%%%%%%%%%%%%%%%%%%%%%%%%%%%%%%%%%%%%%%%%%%%%%%%%%%%

\begin{cor} \label{corl2}
 With the assumptions of the previous theorem. Let $p_{4}$ in
 $\C\parent{\triangle x_{1}x_{2}x_{3}}$, and \mbox{$d_{1}$, $d_{2}$ and $d_{3}$} be the
 midpoints sides of the $p_{4}$-Antitriangle of the triangle $\triangle x_{1}x_{2}x_{3}$
 \mbox{respectively}, then the following holds:
 \begin{enumerate}
  \item $\llaves{x_{1},x_{2},x_{3},x_{4}}$ and $\llaves{p_{1},p_{2},p_{3},p_{4}}$ are
  $\C$-orthocentrics systems.
  \item The points $d_{1}$, $d_{2}$ and $d_{3}$ are the midpoints of the segments that join
  the vertices of the \mbox{triangle} $\triangle x_{1}x_{2}x_{3}$ with its $\C$-orthocenter.
  \item The midpoints sides of the triangle $\triangle x_{1}x_{2}x_{3}$ and its
  $p_{4}$-Antitriangle, lies in the circle with center $q$ and radio $\norm{q-m_{1}}$
  (The Six Points Circle).
 \end{enumerate}
\end{cor}
\begin{proof}
 (1) Since $p_{4}$ is the circumcenter of the triangle $\triangle x_{1}x_{2}x_{3}$ and by
 the construction used in the Theorem (\ref{teor1}) we get that $x_{4}=S_{q}(p_{4})$, so
 $x_{4}$ is the $\C$-orthocenter of the triangle $\triangle x_{1}x_{2}x_{3}$. Furthermore,
 since $p_{i}$ is the circumcenter of the triangle $\triangle x_{j}x_{k}x_{j}$ for
 $\{i,j,k,l\}=\{1,2,3,4\}$, and $x_{i}=S_{q}(p_{i})$ for $i=1,2,3$, then the sets
 $\llaves{x_{1},x_{2},x_{3},x_{4}}$ and $\llaves{p_{1},p_{2},p_{3},p_{4}}$ are
 $\C$-orthocentric systems (see \textit{Figure \ref{FG11}}).
 
\begin{figure}[h!]
 \begin{center}
  \includegraphics[scale=0.32]{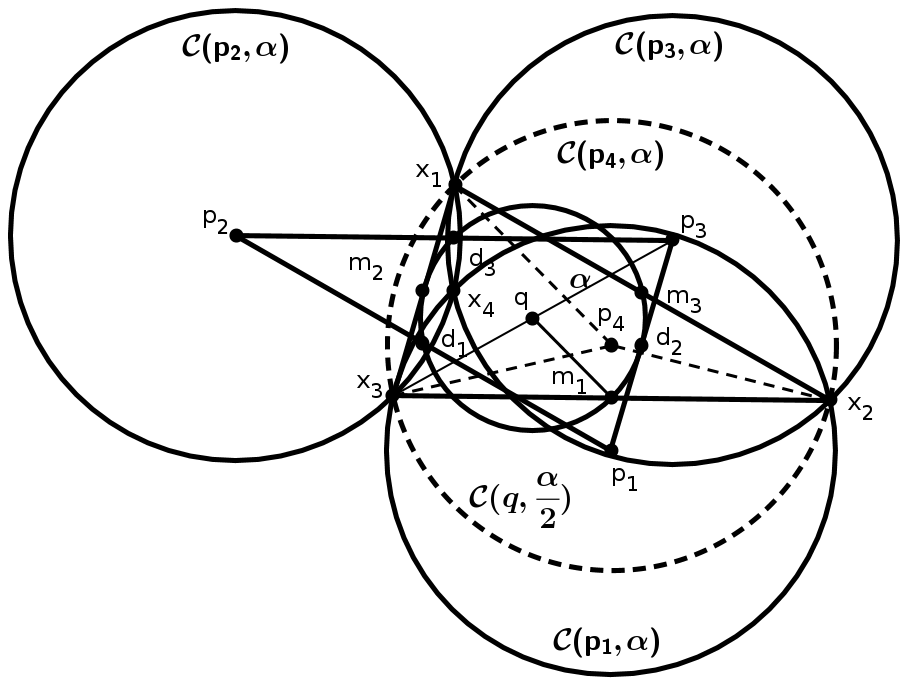}
  \caption{\label{FG11} \textit{Six Points Circle}}
 \end{center}
\end{figure}

(2) Since $d_{i}=\frac{p_{j}+p_{k}}{2}$ for $\{i,j,k\}=\{1,2,3\}$ and $p_{i}=S_{m_{i}}(p_{4})$,
then
 \begin{equation*}
  d_{i}=\frac{p_{j}+p_{k}}{2}=\frac{2m_{j}-p_{4}+2m_{k}-p_{4}}{2}=
  \frac{2x_{i}+x_{k}+x_{j}-2p_{4}}{2}=\frac{x_{i}+x_{4}}{2}
 \end{equation*}
for $\{i,j,k\}=\{1,2,3\}$.

(3) Since $p_{4}$ is the circumcenter of the triangle $\triangle x_{1}x_{2}x_{3}$, then
 \begin{equation*}
  \norm{p_{4}-x_{i}}=\alpha
 \end{equation*}
 for $i=1,2,3$. Now, to prove that $\C(q,\norm{q-m_{1}})$ passes through the midpoints side
 of the triangles $\triangle x_{1}x_{2}x_{3}$ and $\triangle p_{1}p_{2}p_{3}$, is enough to
 see that
 \begin{equation*}
  \norm{q-m_{i}}=\alpha\hspace{1cm}\text{y}\hspace{1cm}\norm{q-d_{i}}=\alpha
 \end{equation*}
 for $i=1,2,3$. Then, notice that
 \begin{equation} \label{eq12}
  \norm{q-m_{i}}=\frac{\norm{p_{4}-x_{i}}}{2}
 \end{equation}
for $i=1,2,3$.
 
On the other hand, $d_{i}=2q-m_{i}$ for $i=1,2,3$ and therefore
\begin{equation} \label{eq15}
 \norm{q-d_{i}}=\norm{q-m_{i}}
\end{equation}
for $i=1,2,3$. Then, by the equations (\ref{eq12}) and (\ref{eq15}) the points
$m_{1}$, $m_{2}$, $m_{3}$, $d_{1}$, $d_{2}$ and $d_{3}$ lie in the circle
$\C(q,\frac{\alpha}{2})$ (see \textit{Figure \ref{FG11}}).
 \end{proof}

In Euclidean Geometry is well known the alinement relation between the barycenter, the
circumcenter and the orthocenter of a triangle. The line passing through these points
is called \textit{Euler's Line} of the triangle. If ``$o$'' is the circumcenter, ``$g$'' the
barycenter and ``$h$'' the orthocenter of a triangle, it is well known that ``$g$'' is an
interior point of the segment $\seg{o,h}$, and if the three points are different, then the
ratio $\frac{hg}{og}=-2$. Also these properties holds in strictly convex Minkowski plane.
The following corollary shows the generalization of the notion of Euler's line in Minkowski
planes using the structure of \mbox{$\C$-orthocenter}, and show that ``$h$'' is
the image of ``$o$'' in the homothety with center ``$g$'' and ratio $-2$. So, the Euler's
Line Theorem in Minkowski planes would be expressed as follows:

%%%%%%%%%%%%%%%%%%%%%%%%%%%%%%%%%%%%%%%%%%%%%%%%%%%%%%%%%%%%%%%%%%%%%%%%%%%%%%%%%%%%%%%%%%%%
%%%%%%%%%%%%%%%%%%%%%%%%%%%%%%%%%%%%%%%% Corolario 3 %%%%%%%%%%%%%%%%%%%%%%%%%%%%%%%%%%%%%%%
%%%%%%%%%%%%%%%%%%%%%%%%%%%%%%%%%%%%%%%%%%%%%%%%%%%%%%%%%%%%%%%%%%%%%%%%%%%%%%%%%%%%%%%%%%%%
 
\begin{cor}
 Let $M$ a Minkowski plane and $x_{1},x_{2},x_{3}$ points of $M$. If
 $g=\frac{x_{1}+x_{2}+x_{3}}{3}$, the barycenter of the triangle $\triangle x_{1}x_{2}x_{3}$,
 then $H_{g,-2}(\C(\triangle x_{1}x_{2}x_{3}))=\Ha(\triangle x_{1}x_{2}x_{3})$.
\end{cor}
\begin{proof}
 Take $p_{4}\in\C(\triangle x_{1}x_{2}x_{3})$, then by (1) from Corollary (\ref{corl2}) there
 exists a point \mbox{$x_{4}\in\Ha(\triangle x_{1}x_{2}x_{3})$} such that
 $\{x_{1},x_{2},x_{3},x_{4}\}$ is a $\C$-orthocentric system. Then, by (4) from Theorem
 (\ref{teor1}) $H_{g,-2}(p_{4})=x_{4}$ and therefore
 \mbox{$H_{g,-2}(\C(\triangle x_{1}x_{2}x_{3}))\subset\Ha(\triangle x_{1}x_{2}x_{3})$}.
 
 Conversely let $x_{4}\in\Ha(\triangle x_{1}x_{2}x_{3})$, then there exists
 $p_{4}\in\C(\triangle x_{1}x_{2}x_{3})$ such that $S_{q}(p_{4})=x_{4}$, with $q$ the point
 of symmetry of the triangles $\triangle x_{1}x_{2}x_{3}$ and its
 \mbox{$p_{4}$-Antitriangle}. Then, $x_{4}$ is the circumcenter of the $p_{4}$-Antitriangle.
 Thus, $p_{4}$ holds that $H_{g,-2}(p_{4})=x_{4}$ and therefore
 $\Ha(\triangle x_{1}x_{2}x_{3})\subset H_{g,-2}(\C(\triangle x_{1}x_{2}x_{3}))$.
\end{proof}

%%%%%%%%%%%%%%%%%%%%%%%%%%%%%%%%%%%%%%%%%%%%%%%%%%%%%%%%%%%%%%%%%%%%%%%%%%%%%%%%%%%%%%%%%%%%
%%%%%%%%%%%%%%%%%%%%%%%%%%%%%%%%%%%% Lema Principal %%%%%%%%%%%%%%%%%%%%%%%%%%%%%%%%%%%%%%%%
%%%%%%%%%%%%%%%%%%%%%%%%%%%%%%%%%%%%%%%%%%%%%%%%%%%%%%%%%%%%%%%%%%%%%%%%%%%%%%%%%%%%%%%%%%%%

The following Lemma will give us the tools to proof the results about \mbox{characterizations}
of euclidean planes by studying geometric properties of $\C$-orthocentric systems in normed
planes.

\begin{lema}\label{le1}
 Let $(M,\norm{\circ})$ be a Minkowski plane, with origen $O$. For any $x,z\in M$ with
 $x\bot_{I}z$, let $p_{3}=-z$, $p_{4}=z$, $x_{1}=x$, $x_{2}=-x$, and $\lambda=\norm{x+z}$.
 Then there are points $x_{3}\in\C(p_{4},\lambda)$ and $q\in\C(O,\lambda/2)$,
 such that $\{p_{1}, p_{2}, p_{3}, p_{4}\}$ and $\llaves{x_{1},x_{2},x_{3},x_{4}}$ are
 $\C$-orthocentric systems, where $p_{1}=S_{q}(x_{1})$, $p_{2}=S_{q}(x_{2})$ and
 $x_{4}=S_{q}(p_{4})$. Furthermore, if $L_{1}=\gen{S_{x_{1}}(p_{3}),S_{x_{2}}(p_{3})}$ and
 $L_{3}$ is the parallel line to $L_{1}$ passing through $p_{4}$, then the following statements
 holds:
 \begin{enumerate}
  \item $p_{1}\in\C(x_{2},\lambda)$, $p_{2}\in\C(x_{1},\lambda)$ and
  $x_{4}\in\C(p_{3},\lambda)$
  \item There are points $p_{1}$ and $p_{2}$ such that
  $\norm{p_{3}-p_{1}}=\norm{p_{3}-p_{2}}$.
  \item $p_{1}=\pm\frac{\lambda z}{\norm{z}}-x$ and $p_{2}=\pm\frac{\lambda z}{\norm{z}}+x$
  if and only if $p_{4}\in\gen{p_{3},\frac{p_{1}+p_{2}}{2}}$. Particularly,
  $p_{1}=\frac{\lambda z}{\norm{z}}-x$ and $p_{2}=\frac{\lambda z}{\norm{z}}+x$ if and only
  if $p_{4}\in[p_{3},\frac{p_{1}+p_{2}}{2}]$.
  \item If $\norm{z}<\lambda$ and
  $q\in Arc^{-}_{\C(O,\frac{\lambda}{2})}\parent{\frac{p_{4}+x_{2}}{2},\frac{p_{4}+x_{1}}{2}}$,
  with $p_{1}\neq p_{4},S_{x_{2}}(p_{3})$, then $p_{3}$ and the line $\gen{p_{1},p_{2}}$ are
  separated by $L_{1}$.
  \item There are points $p_{1}$ and $p_{2}$ such that the line $\gen{p_{1},p_{2}}$ and
  $p_{3}$ are separated by $L_{1}$, or \mbox{$\gen{p_{1},p_{2}}=L_{1}$}. Furthermore
  $p_{4}\in A_{B}\left([p_{3},p_{1}\rangle,[p_{3},p_{2}\rangle\right)$.
 \end{enumerate}
\end{lema}
\begin{proof}
 Takes the circle $\C(p_{4},\lambda)$. If $x_{3}\in \C(p_{4},\lambda)$, then
 $p_{4}$ is the cincumcenter of the \mbox{triangle} $\triangle x_{1}x_{2}x_{3}$. By (1) from
 Theorem (\ref{teor1}), $q=\frac{p_{3}+x_{3}}{2}$ and \mbox{therefore} $p_{1}=S_{q}(x_{1})$,
 $p_{2}=S_{q}(x_{2})$ and $x_{4}=S_{q}(p_{4})$. So, by (1) from Corollary (\ref{corl2}), the
 sets $\llaves{p_{1},p_{2},p_{3},p_{3}}$ and $\llaves{x_{1},x_{2},x_{3},x_{4}}$ are
 \mbox{$\C$-orthocentric} systems.
 
\begin{figure}[h!]
 \begin{center}
  \includegraphics[scale=0.17]{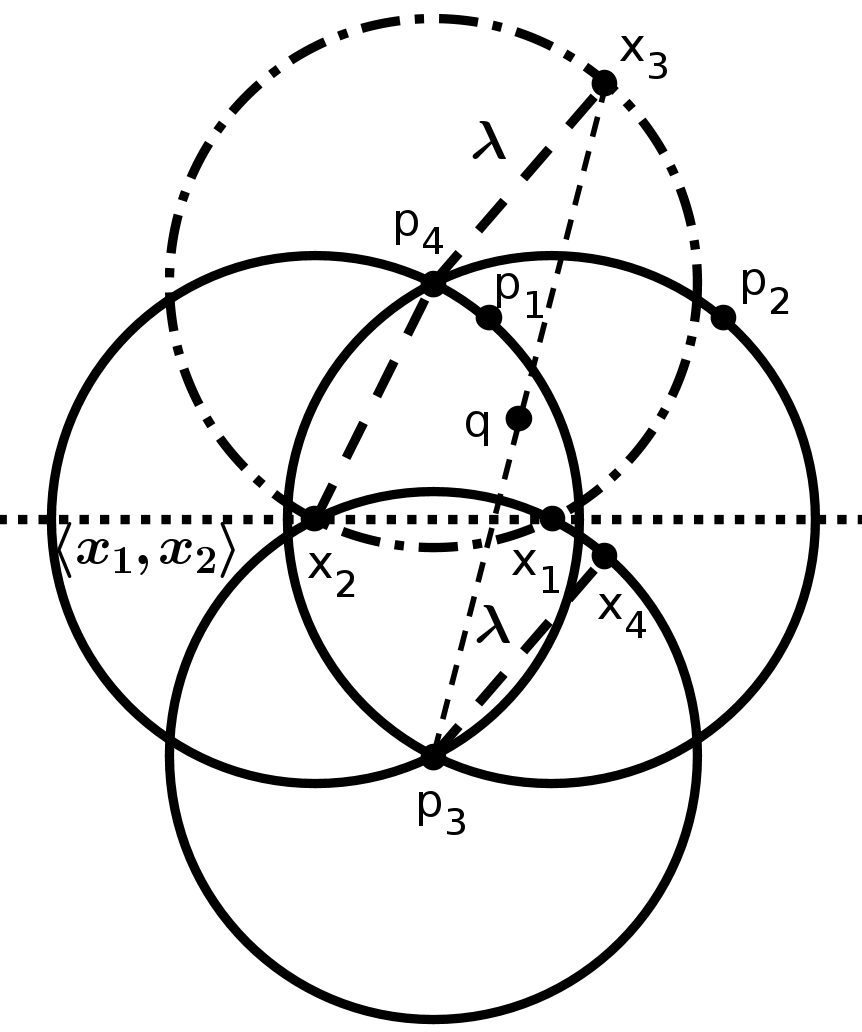}
  \caption{\label{FG12} \textit{Circles} $\C(x_{4},\lambda)$ and $\C(p_{4},\lambda)$}
 \end{center}
\end{figure}
 
(1) To prove that $p_{1}\in\C(x_{2},\lambda)$, $p_{2}\in\C(x_{1},\lambda)$ and
$x_{4}\in\C(p_{3},\lambda)$ we have to see that
\begin{equation*}
 \norm{x_{1}-p_{2}}=\norm{x_{2}-p_{1}}=\norm{p_{3}-x_{4}}=\lambda
\end{equation*}
Note that $\norm{p_{4}-x_{3}}=\lambda$, then $\lambda=2\norm{q}$ so that
\begin{equation} \label{eq2.1}
 \norm{q}=\frac{\lambda}{2}
\end{equation}
Since $p_{2}=S_{q}(x_{2})$, $p_{1}=S_{q}(x_{1})$, $x_{4}=S_{q}(p_{4})$ (see
\mbox{\textit{Figure \ref{FG12}}}), and ussing (\ref{eq2.1}), then 
\begin{eqnarray*}
 \norm{x_{1}-p_{2}}=\norm{x_{2}-p_{1}}=\norm{p_{3}-x_{4}}=2\norm{q}=\lambda
\end{eqnarray*}

(2) Without loss of generality takes the arc $Arc^{+}_{\C(p_{4},\lambda)}\parent{x_{1},x_{2}}$,
which is a closed set (see \textit{Figure \ref{FG12}}). Define the continuous parameterization
\mbox{$\alpha:[0,1]\longrightarrow Arc^{+}_{\C(p_{4},\lambda)}\parent{x_{1},x_{2}}$}, such
that
\begin{equation} \label{eq2.2}
\lim_{\substack{t\rightarrow 0}}\;\alpha(t)=x_{1}\hspace{2cm}
\lim_{\substack{t\rightarrow 1}}\;\alpha(t)=x_{2}
\end{equation}

Define the continuous function
$f:Arc^{+}_{\C(p_{4},\lambda)}\parent{x_{1},x_{2}}\longrightarrow\ere$ by:
\begin{equation} \label{eq2.3}
 f(x)=\norm{x_{2}-x}-\norm{x_{1}-x}
\end{equation}

\begin{figure}[h!]
 \begin{center}
  \includegraphics[scale=0.23]{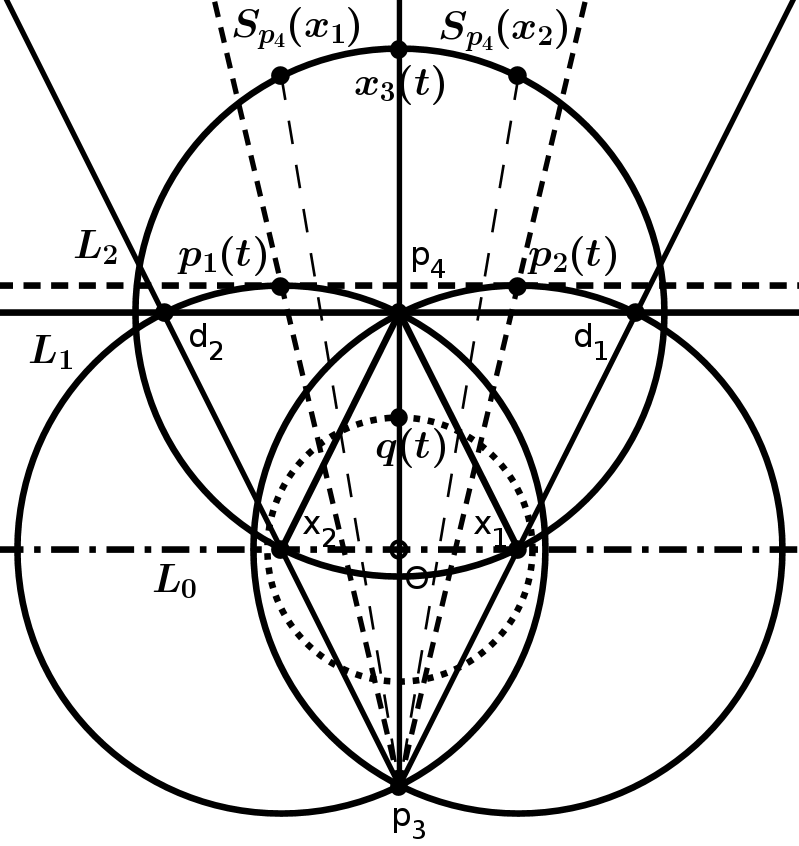}\hspace{3mm}\includegraphics[scale=0.16]{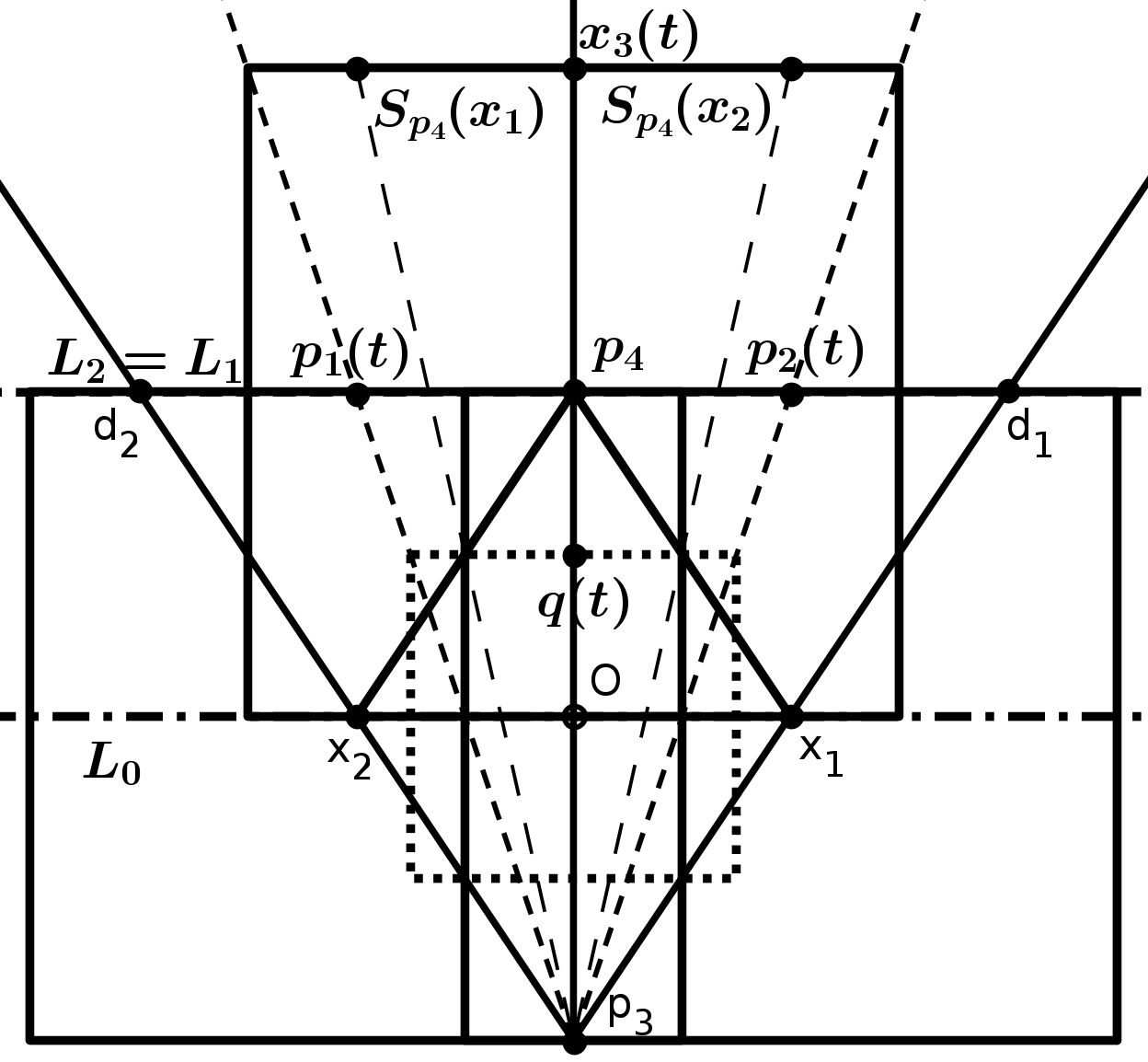}
  \caption{\label{FG13} \textit{Demonstration (3) and (5), Lemma (\ref{le1})}}
 \end{center}
\end{figure}

On the other hand, for any $x\in Arc^{+}_{\C(p_{4},\lambda)}\parent{x_{1},x_{2}}$ there
exists $t\in[0,1]$ such that $\alpha(t)=x$, then by (\ref{eq2.2}) and (\ref{eq2.3})
$f(\alpha(t))$ holds:
\begin{equation*}
\lim_{\substack{t\rightarrow 0}}\;f(\alpha(t))=\norm{x_{2}-x_{1}}>0\hspace{1cm}\text{and}
\hspace{1cm}\lim_{\substack{t\rightarrow 1}}\;f(\alpha(t))=-\norm{x_{1}-x_{2}}<0
\end{equation*}
Then, since $f$ is defined on a compact set, taking negative and positive values, there
exists $x_{3}\in Arc^{+}_{\C(p_{4},\lambda)}\parent{x_{1},x_{2}}$ such that $f(x_{3})=0$,
i.e, a $x_{3}$ that holds
\begin{equation} \label{eq2.4}
 \norm{x_{1}-x_{3}}=\norm{x_{2}-x_{3}}
\end{equation}

Now, define the points $q=\frac{x_{3}+p_{3}}{2}$, $p_{1}=S_{q}(x_{1})$ and
$p_{2}=S_{q}(x_{2})$. By (\ref{eq2.4}), and the fact that $S_{q}(x)$ is a isometry, then
$\norm{p_{3}-p_{1}}=\norm{p_{3}-p_{2}}$.

(3) To get that $p_{4}\in\gen{p_{3},\frac{p_{1}+p_{2}}{2}}$, since $p_{4}=z$ and
$p_{3}=-z$, $\frac{p_{1}+p_{2}}{2}$ need to be of the form $kz$, with $k\in\ere$ (see
\textit{Figure \ref{FG13}}). Since $\frac{p_{1}+p_{2}}{2}=2q$, then $q=\frac{kz}{2}$ and by
the equation (\ref{eq2.1})
\begin{eqnarray*}
 \frac{\lambda}{2}=\norm{q}=\norm{\frac{kz}{2}}
\end{eqnarray*}
therefore $k=\pm\frac{\lambda}{\norm{z}}$. So, $q=\pm\frac{\lambda z}{2\norm{z}}$ are the
only possible values for which the \mbox{condition} is satisfied. Then, the only pair of
points $\llaves{p_{1},p_{2}}$ would be
 $\llaves{\frac{\lambda z}{\norm{z}}-x,\frac{\lambda z}{\norm{z}}+x}$ and
 $\llaves{-\frac{\lambda z}{\norm{z}}-x,-\frac{\lambda z}{\norm{z}}+x}$.
 
(4) Since $p_{1}\notin\{p_{4},S_{x_{2}}(p_{3})\}$, $\norm{z}<\lambda$, and by the equation
(\ref{eq2.1}), we have that $p_{4}\neq\frac{p_{1}+p_{2}}{2}=2q$. Furthermore, the lines
$\gen{p_{1},p_{2}}$ and $\gen{S_{x_{1}}(p_{3}),S_{x_{2}}(p_{3})}$ are different and they are
paralles because
\begin{equation}\label{eq2.9}
S_{x_{1}}(p_{3})-S_{x_{2}}(p_{3})=2(p_{2}-p_{1})
\end{equation}
Now, if
$q\in Arc^{-}_{\C(O,\frac{\lambda}{2})}\parent{\frac{p_{4}+x_{2}}{2},\frac{p_{4}+x_{1}}{2}}$,
then $p_{1}$ lies in the $Arc^{-}_{\C(x_{2},\lambda)}\parent{S_{x_{2}}(p_{3}),p_{4}}$.
and $p_{2}$ lies in the $Arc^{-}_{\C(x_{1},\lambda)}\parent{p_{4},S_{x_{1}}(p_{3})}$
(see \textit{Figure \ref{FG14}}). Also, $\frac{p_{1}+p_{2}}{2}$ and $p_{3}$ are separated
by $L_{1}=\gen{S_{x_{2}}(p_{3}),S_{x_{1}}(p_{3})}$. Since
$\frac{p_{1}+p_{2}}{2}\in\seg{p_{1},p_{2}}$ and the equation (\ref{eq2.9}), then
$\gen{p_{1},p_{2}}$ and $p_{3}$ are separated by $L_{1}$.

\begin{figure}[h!]
 \begin{center}
  \includegraphics[scale=0.28]{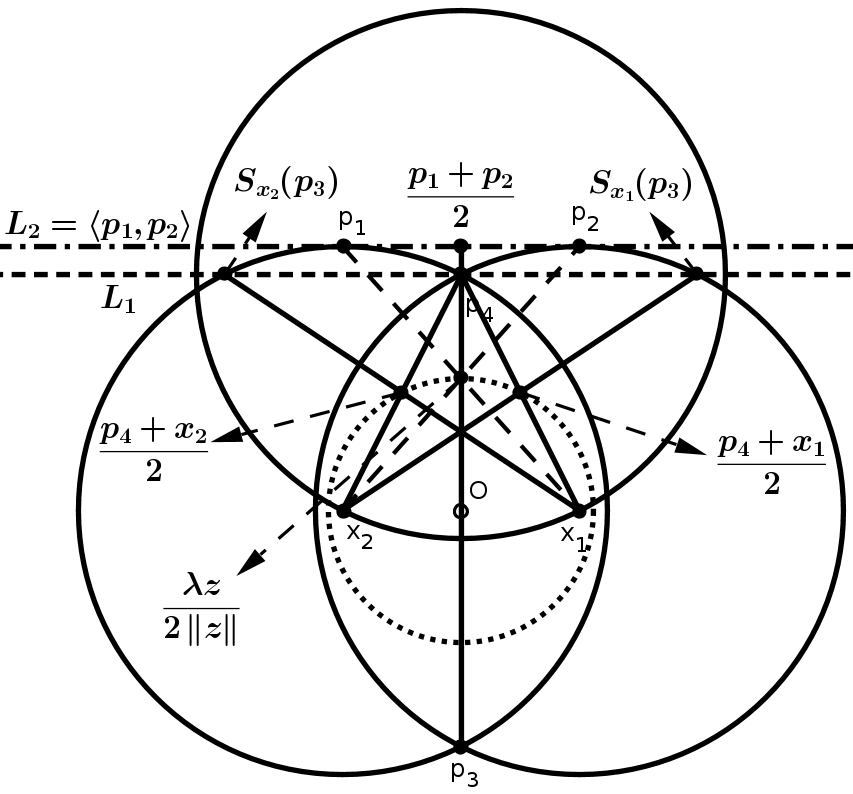}
  \caption{\label{FG14} \textit{Lines $\gen{p_{1},p_{2}}$ and
  $\gen{S_{x_{1}}(p_{3}),S_{x_{2}}(p_{3})}$}}
 \end{center}
\end{figure}

(5) For any $t\in(0,1)$, define the continuous functions
$$x_{3}(t):[0,1]\longrightarrow Arc^{+}_{\C(p_{4},\lambda)}\parent{S_{p_{4}}(x_{1}),S_{p_{4}}(x_{2})}$$
$q(t)=\frac{p_{3}+x_{3}(t)}{2}$ and $p_{i}(t)=2q(t)-x_{i}$ for $i=1,2$. When $t$ is moving
from $0$ to $1$, the ray $A_{B}\parent{\rayo{p_{3},p_{1}(t)},\rayo{p_{3},p_{2}(t)}}$
continually goes from $A_{B}\parent{\rayo{p_{3},d_{2}},\rayo{p_{3},p_{4}}}$ to
$A_{B}\parent{\rayo{p_{3},p_{4}},\rayo{p_{3},d_{1}}}$ (see \textit{Figure \ref{FG13}}).
So, there exists $t_{0}\in(0,1)$ such that
\begin{equation*}
 A_{B}\parent{\rayo{p_{3},p_{1}(t_{0})},\rayo{p_{3},p_{2}(t_{0})}}=\rayo{p_{3},p_{4}}
\end{equation*}
Now, let $p_{1}=p_{1}(t_{0})$ and $p_{2}=p_{2}(t_{0})$, then $p_{1}$ and $p_{2}$ have the
desired property. By (4) from this Lemma, $p_{3}$ and the line $\gen{p_{1},p_{2}}$ are
separated by the line $L_{1}$ passing through $p_{4}$ parallel to $\gen{p_{1},p_{2}}$, or
$\gen{p_{1},p_{2}}=L_{1}$.
\end{proof}

%%%%%%%%%%%%%%%%%%%%%%%%%%%%%%%%%%%%%%%%%%%%%%%%%%%%%%%%%%%%%%%%%%%%%%%%%%%%%%%%%%%%%%%%%%%%
%%%%%%%%%%%%%%%%%%%%%%%%%%%%%% Caracterización de Euclidianidad %%%%%%%%%%%%%%%%%%%%%%%%%%%%
%%%%%%%%%%%%%%%%%%%%%%%%%%%%%%%%%%%%%%%%%%%%%%%%%%%%%%%%%%%%%%%%%%%%%%%%%%%%%%%%%%%%%%%%%%%%

\begin{teor}\label{teor2}
 A Minkowski plane $(M,\norm{\circ})$ is euclidean if and only if for any
 \mbox{$\C$-orthocentric} system $\llaves{p_{1},p_{2},p_{3},p_{4}}$ the relation
 \begin{equation*}
  (p_{i}-p_{j})\bot_{B}(p_{k}-p_{l})
 \end{equation*}
holds, where $\llaves{i,j,k,l}=\llaves{1,2,3,4}$
\end{teor}
\begin{proof}
If $(M,\norm{\circ})$ is euclidean, then for any $\C$-orthocentric system
$\llaves{p_{1},p_{2},p_{3},p_{4}}$, $p_{i}$ is the $\C$-orthocenter of
$\triangle p_{j}p_{k}p_{l}$ for $\llaves{i,j,k,l}=\llaves{1,2,3,4}$. Furthermore
 \begin{equation*}
  (p_{i}-p_{j})\bot_{B}(p_{k}-p_{l})
 \end{equation*}

Conversely, for any $x,y\in M$ with $x\bot_{I}y$, let
\begin{equation*}
 p_{4}=y \hspace{1cm} p_{3}=-y \hspace{1cm} x_{1}=x \hspace{1cm} x_{2}=-x
\end{equation*}
By Lemma (\ref{le1}), there are two points $p_{1}$ and $p_{2}$ such that
$\llaves{p_{1},p_{2},p_{3},p_{4}}$ is a \mbox{$\C$-orthocentric} system (see
\textit{Figure \ref{FG12}}). By Theorem (\ref{teor1}) we have
\begin{equation*}
 p_{2}-p_{1}=x_{1}-x_{2}=2x
\end{equation*}
Since $(p_{2}-p_{1})\bot_{B}(p_{4}-p_{3})$ and $p_{4}-p_{3}=2y$, then $x\bot_{B}y$.
So, by Lemma (\ref{lepre2}), $(M,\norm{\circ})$ is euclidean.
\end{proof}

%%%%%%%%%%%%%%%%%%%%%%%%%%%%%%%%%%%%%%%%%%%%%%%%%%%%%%%%%%%%%%%%%%%%%%%%%%%%%%%%%%%%%%%%%%%%
%%%%%%%%%%%%%%%%%%%%%%%%%%%%%%%%%%%%%% Teorema 2 %%%%%%%%%%%%%%%%%%%%%%%%%%%%%%%%%%%%%%%%%%%
%%%%%%%%%%%%%%%%%%%%%%%%%%%%%%%%%%%%%%%%%%%%%%%%%%%%%%%%%%%%%%%%%%%%%%%%%%%%%%%%%%%%%%%%%%%%

\begin{teor}\label{teor3}
 A Minkowski plane $(M,\norm{\circ})$ is euclidean if and only if for any
 \mbox{$\C$-orthocentric} system $\llaves{p_{1},p_{2},p_{3},p_{4}}$, with
 $\norm{p_{3}-p_{1}}=\norm{p_{3}-p_{2}}$, it holds that
 \mbox{$p_{4}\in\gen{p_{3},\frac{p_{1}+p_{2}}{2}}$}.
\end{teor}
\begin{proof}
 If $(M,\norm{\circ})$ is euclidean, any $\C$-orthocentric system holds that
 $p_{4}\in\gen{p_{3},\frac{p_{1}+p_{2}}{2}}$ if $\norm{p_{3}-p_{1}}=\norm{p_{3}-p_{2}}$.

 Conversely, let $x,y\in(M,\norm{\circ})$ such that $x\bot_{I}y$, by Lemma (\ref{lepre1}) is
 enough to prove that there exists a $t>1$ such that $x\bot_{I}ty$. Clearly it exists
 such $t$ if $x=O$ or $y=O$. Now, let $x,y\in M\smallsetminus\llaves{O}$ with $x\bot_{I}y$.
 Define $\lambda=\norm{x+y}$ and
 \begin{equation*}
  p_{4}=y\hspace{1cm}p_{3}=-y\hspace{1cm}x_{1}=x\hspace{1cm}x_{2}=-x
 \end{equation*}
then by (2) from Lemma (\ref{le1}) we have points $p_{1}$ and $p_{2}$ such that
\begin{equation}\label{eq2.10}
 \norm{p_{3}-p_{1}}=\norm{p_{3}-p_{2}}
\end{equation}
By hypothesis $p_{4}\in\gen{p_{3},\frac{p_{1}+p_{2}}{2}}$, and ussing (3) from Lemma
(\ref{le1}), we can take $p_{1}=\frac{\lambda y}{\norm{y}}-x$ and
$p_{2}=\frac{\lambda y}{\norm{y}}+x$. So,
\begin{equation*}
 \norm{p_{3}-p_{1}}=\norm{x-\parent{1+\frac{\lambda}{\norm{y}}}y}
\end{equation*}
\begin{equation*}
 \norm{p_{3}-p_{2}}=\norm{x+\parent{1+\frac{\lambda}{\norm{y}}}y}
\end{equation*}
and by the equation (\ref{eq2.10}), then $x\bot_{I}ty$ where
$t=1+\frac{\lambda}{\norm{y}}>1$, since $\lambda>0$.
\end{proof}

%%%%%%%%%%%%%%%%%%%%%%%%%%%%%%%%%%%%%%%%%%%%%%%%%%%%%%%%%%%%%%%%%%%%%%%%%%%%%%%%%%%%%%%%%%%%
%%%%%%%%%%%%%%%%%%%%%%%%%%%%%%%%%%%%%%%%% Teorema 3 %%%%%%%%%%%%%%%%%%%%%%%%%%%%%%%%%%%%%%%%
%%%%%%%%%%%%%%%%%%%%%%%%%%%%%%%%%%%%%%%%%%%%%%%%%%%%%%%%%%%%%%%%%%%%%%%%%%%%%%%%%%%%%%%%%%%%

\begin{teor}\label{teor4}
 A Minkowski plane $(M,\norm{\circ})$ is euclidean if and only if for any
 \mbox{$\C$-orthocentric} system $\llaves{p_{1},p_{2},p_{3},p_{4}}$, the equality
 $\norm{p_{3}-p_{1}}=\norm{p_{3}-p_{2}}$ holds, \mbox{whenever}
 $p_{4}\in\gen{p_{3},\frac{p_{1}+p_{2}}{2}}$.
\end{teor}
\begin{proof}
If $(M,\norm{\circ})$ is euclidean, any $\C$-orthocentric system, with
$p_{4}\in\gen{p_{3},\frac{p_{1}+p_{2}}{2}}$, holds that
$\norm{p_{3}-p_{1}}=\norm{p_{3}-p_{2}}$.
 
Conversely, let $x,y\in(M,\norm{\circ})$ with $x\bot_{I}y$, by Lemma (\ref{lepre1}) is enough
to prove that there exists $t>1$ such that $x\bot_{I}ty$. Clearly it exists such $t$ if
$x=O$ or $y=O$. Now, let $x,y\in M\smallsetminus\llaves{O}$ with $x\bot_{I}y$. Define
$\lambda=\norm{x+y}$ and
 \begin{equation*}
  p_{4}=y\hspace{1cm}p_{3}=-y\hspace{1cm}x_{1}=x\hspace{1cm}x_{2}=-x
 \end{equation*}
then by (3) from Lemma (\ref{le1}) we have the points $p_{1}=\pm\frac{\lambda y}{\norm{y}}-x$
and $p_{2}=\pm\frac{\lambda y}{\norm{y}}+x$. Take the values
$p_{1}=\frac{\lambda y}{\norm{y}}-x$ and $p_{2}=\frac{\lambda y}{\norm{y}}+x$, and since
$\norm{p_{3}-p_{1}}=\norm{p_{3}-p_{2}}$, then $x\bot_{I}ty$ with
$t=1+\frac{\lambda}{\norm{y}}>1$.
\end{proof}

%%%%%%%%%%%%%%%%%%%%%%%%%%%%%%%%%%%%%%%%%%%%%%%%%%%%%%%%%%%%%%%%%%%%%%%%%%%%%%%%%%%%%%%%%%%%
%%%%%%%%%%%%%%%%%%%%%%%%%%%%%%%%%%%%%%%%% Teorema 4 %%%%%%%%%%%%%%%%%%%%%%%%%%%%%%%%%%%%%%%%
%%%%%%%%%%%%%%%%%%%%%%%%%%%%%%%%%%%%%%%%%%%%%%%%%%%%%%%%%%%%%%%%%%%%%%%%%%%%%%%%%%%%%%%%%%%%

\begin{teor}\label{teor5}
 A Minkowski plane $\parent{M,\norm{\circ}}$ is euclidean if and only if for any
 \mbox{$\C$-orthocentric} system $\llaves{p_{1},p_{2},p_{3},p_{4}}$, $p_{4}$ lies on the line
 containing $A_{B}\parent{\rayo{p_{3},p_{1}},\rayo{p_{3},p_{2}}}$ whenever
 $\norm{p_{3}-p_{1}}=\norm{p_{3}-p_{2}}$.
\end{teor}
\begin{proof}
 If $(M,\norm{\circ})$ is euclidean, any $\C$-orthocentric system with
 \mbox{$\norm{p_{3}-p_{1}}=\norm{p_{3}-p_{2}}$}, holds that $p_{4}$ lies on the line
 containing $A_{B}\parent{\rayo{p_{3},p_{1}},\rayo{p_{3},p_{2}}}$.

Conversely, let $x,y\in(M,\norm{\circ})$ with $x\bot_{I}y$. By Theorem (\ref{teor3}), we
only have to proof that for any $\C$-orthocentric system $\llaves{p_{1},p_{2},p_{3},p_{4}}$,
$p_{4}\in\gen{p_{3},\frac{p_{1}+p_{2}}{2}}$ whenever
\mbox{$\norm{p_{3}-p_{1}}=\norm{p_{3}-p_{2}}$}. By the definition of Busemann angular
bisectors and the fact that \mbox{$\norm{p_{3}-p_{1}}=\norm{p_{3}-p_{2}}$}, then
\begin{equation*}
 A_{B}\parent{\rayo{p_{3},p_{1}},\rayo{p_{3},p_{2}}}=\rayo{p_{3},\frac{p_{1}+p_{2}}{2}}
\end{equation*}
Thus $\gen{p_{3},\frac{p_{1}+p_{2}}{2}}$ is the line containing
$A_{B}\parent{\rayo{p_{3},p_{1}}\rayo{p_{3},p_{2}}}$ and therefore
$p_{4}\in\gen{p_{3},\frac{p_{1}+p_{2}}{2}}$.
\end{proof}

%%%%%%%%%%%%%%%%%%%%%%%%%%%%%%%%%%%%%%%%%%%%%%%%%%%%%%%%%%%%%%%%%%%%%%%%%%%%%%%%%%%%%%%%%%%%
%%%%%%%%%%%%%%%%%%%%%%%%%%%%%%%%%%%%%%%%% Teorema 5 %%%%%%%%%%%%%%%%%%%%%%%%%%%%%%%%%%%%%%%%
%%%%%%%%%%%%%%%%%%%%%%%%%%%%%%%%%%%%%%%%%%%%%%%%%%%%%%%%%%%%%%%%%%%%%%%%%%%%%%%%%%%%%%%%%%%%

\begin{teor}\label{teor5}
 A Minkowski plane $\parent{M,\norm{\circ}}$ is euclidean if and only if for any
 \mbox{$\C$-orthocentric} system $\llaves{p_{1},p_{2},p_{3},p_{4}}$, the equality
 $\norm{p_{3}-p_{1}}=\norm{p_{3}-p_{2}}$ holds whenever $p_{4}$ lies on the line containing
 $A_{B}\parent{\rayo{p_{3},p_{1}},\rayo{p_{3},p_{2}}}$.
\end{teor}
\begin{proof}
If $(M,\norm{\circ})$ is euclidean, any $\C$-orthocentric system
$\llaves{p_{1},p_{2},p_{3},p_{4}}$, with $p_{4}$ on the line containing
$A_{B}\parent{\rayo{p_{3},p_{1}},\rayo{p_{3},p_{2}}}$, holds the equality
$\norm{p_{3}-p_{1}}=\norm{p_{3}-p_{2}}$.
 
Conversely, let $x,y\in(M,\norm{\circ})$ with $x\bot_{I}y$. By the Lemma (\ref{lepre1}) we only
have to prove that there exist $t>1$ such that $x\bot_{I}ty$. Clearly it exists such $t$
if $x=O$ or $y=O$. Now, let $x,y\in M\smallsetminus\llaves{O}$ with $x\bot_{I}y$. Define
$\lambda=\norm{x+y}$ and
 \begin{equation*}
  p_{4}=y\hspace{1cm}p_{3}=-y\hspace{1cm}x_{1}=x\hspace{1cm}x_{2}=-x
 \end{equation*}
By (5) from Lemma (\ref{le1}) we have a $\C$-orthocentric system
$\llaves{p_{1},p_{2},p_{3},p_{4}}$ such that
\mbox{$p_{4}\in A_{B}\parent{\rayo{p_{3},p_{1}},\rayo{p_{3},p_{2}}}$}. Since
$\norm{p_{3}-p_{1}}=\norm{p_{3}-p_{2}}$, and the definition of Busemann angular
\mbox{bisectors}, then
\begin{equation*}
 A_{B}\parent{\rayo{p_{3},p_{1}},\rayo{p_{3},p_{2}}}=\rayo{p_{3},\frac{p_{1}+p_{2}}{2}}
\end{equation*}

Thus, by (3) from Lemma (\ref{le1}), we have $p_{1}=\frac{\lambda y}{\norm{y}}-x$ and
$p_{2}=\frac{\lambda y}{\norm{y}}+x$. Then,
\begin{equation*}
 \norm{x-\parent{1+\frac{\lambda}{\norm{y}}}y}=\norm{p_{3}-p_{1}}=\norm{p_{3}-p_{2}}=
 \norm{x+\parent{1+\frac{\lambda}{\norm{y}}}y}
\end{equation*}
and therefore $x\bot_{I}ty$, taking $t=1+\frac{\lambda}{\norm{y}}>1$.
\end{proof}

%%%%%%%%%%%%%%%%%%%%%%%%%%%%%%%%%%%%%%%%%%%%%%%%%%%%%%%%%%%%%%%%%%%%%%%%%%%%%%%%%%%%%%%%%%%%
%%%%%%%%%%%%%%%%%%%%%%%%%%%%%%%%%%%%%%% Bibliografía %%%%%%%%%%%%%%%%%%%%%%%%%%%%%%%%%%%%%%%
%%%%%%%%%%%%%%%%%%%%%%%%%%%%%%%%%%%%%%%%%%%%%%%%%%%%%%%%%%%%%%%%%%%%%%%%%%%%%%%%%%%%%%%%%%%%

\end{document}